\documentclass[12pt, reqno]{amsart}
\usepackage{amsmath, amsthm, amscd, amsfonts, amssymb, graphicx, color}
\usepackage[bookmarksnumbered, colorlinks, plainpages]{hyperref}
\hypersetup{colorlinks=true,linkcolor=red, anchorcolor=green, citecolor=cyan, urlcolor=red, filecolor=magenta, pdftoolbar=true}

\newtheorem{theorem}{Theorem}[section]
\newtheorem{lemma}[theorem]{Lemma}
\newtheorem{proposition}[theorem]{Proposition}
\newtheorem{corollary}[theorem]{Corollary}
\theoremstyle{definition}
\newtheorem{definition}[theorem]{Definition}

\theoremstyle{remark}

\numberwithin{equation}{section}

\begin{document}
 
\title[Inequalities in quantum information theory]{Noncommutative versions of inequalities in quantum information theory}

\author{Ali Dadkhah}
\address{Department of Pure Mathematics, Ferdowsi University of Mashhad, P. O. Box 1159, Mashhad 91775, Iran}
\email{dadkhah61@yahoo.com}
 
\author{ Mohammad Sal Moslehian}
\address{Department of Pure Mathematics, Center Of Excellence in Analysis on Algebraic Structures (CEAAS), Ferdowsi University of Mashhad, P. O. Box 1159, Mashhad 91775, Iran}
\email{moslehian@um.ac.ir}
 
\author{Kenjiro Yanagi}
\address{Department of Mathematics, Josai University, Sakado, 350-0295, Japan}
\email{yanagi@josai.ac.jp}
 

\begin{abstract} 
In this paper, we aim to replace in the definitions of covariance and correlation the usual trace {\rm Tr} by a tracial positive map between unital $C^*$-algebras and to replace the functions $x^{\alpha}$ and $x^{1-\alpha}$ by functions $f$ and $g$ satisfying some mild conditions. These allow us to define the generalized covariance, the generalized variance, the generalized correlation and the generalized Wigner--Yanase--Dyson skew information related to the tracial positive maps and functions $f$ and $g$. We present a generalization of Heisenberg's uncertainty relation in the noncommutative framework. 
We extend some inequalities and properties for the generalized correlation and the generalized Wigner--Yanase--Dyson skew information. Furthermore, we extend some inequalities for the generalized skew information such as uncertainty relation and the relation between the generalized variance and the generalized skew information.
\end{abstract}

\maketitle
\section{Introduction}
In quantum information theory, the classical expectation value of
an observable (self-adjoint element) $A$ in a quantum state
(density element) $\rho$ is defined by ${\rm Tr}(\rho A)$, and the classical variance
 is expressed by
$ {\rm Var}_\rho(A) := {\rm Tr}(\rho A^2) -({\rm Tr}(\rho A))^2$.
The Heisenberg uncertainty relation
\cite{fried,yoo} states that
\begin{eqnarray}\label{heis1}
{\rm Var}_{\rho}(A){\rm Var}_{\rho}(B) \geq \dfrac{1}{4} \left|{\rm Tr}\left(\rho [A,B]\right)\right|^2,
\end{eqnarray}
where $\rho$ is a quantum state and $A$ and $B$ are two observables. The Heisenberg uncertainty relation gives a fundamental limit for the measurements of incompatible observables. A refinement of the Heisenberg uncertainty relation due to  Schr\"{o}dinger \cite{schro} is given by
\begin{eqnarray}\label{schr}
{\rm Var}_{\rho}(A){\rm Var}_{\rho}(B) -\left|{\rm Re} ({\rm Cov}_{\rho}(A,B))\right|^2\geq \frac{1}{4} \left|{\rm Tr}\left(\rho [A,B]\right)\right|^2,
\end{eqnarray}
where $[A,B]:=AB-BA$ is the commutator of $A$ and $B$ and the classical covariance ${\rm Cov}_{\rho}(A,B)$ of $A$ and $B$ is defined by ${\rm Cov}_{\rho}(A,B) := {\rm Tr}(\rho AB) -{\rm Tr}(\rho A) {\rm Tr}(\rho B)$.

The third author, Furuichi, and Kuriyama \cite{yangi} defined the one-parameter correlation and the one-parameter Wigner--Yanase skew information (is known as the Wigner--Yanase--Dyson skew information; see \cite{modad2,lib}) for elements $A$ and $B$, respectively, as follows:
\begin{eqnarray*}
{\rm Corr}_{\rho}^{\alpha}(A,B):={\rm Tr}(\rho A^*B)-{\rm Tr}(\rho^{1-\alpha} A^*\rho^{ \alpha}B) \mbox{\quad and \quad} {\rm I}_{\rho}^{\alpha}(A):={\rm Corr}_{\rho}^{\alpha}(A,A),
\end{eqnarray*}
where $\alpha \in [0,1]$. They showed a trace inequality representing the relation between these two quantities as
\begin{eqnarray}\label{semic}
\left| {\rm Re(Corr}_\rho^{\alpha} (A,B))\right|^2 \leq {\rm I}_{\rho}^\alpha (A){\rm I}_{\rho}^\alpha (B).
\end{eqnarray}
In the case that $\alpha =\frac{1}{2}$, we get the classical notions of the correlation ${\rm Corr}_{\rho}(A,B)$ and the Wigner--Yanase skew information ${\rm I}_{\rho}(A)$.

In this paper, we aim to replace the usual trace {\rm Tr} by a tracial positive map between unital $C^*$-algebras and to replace the functions $x^\alpha$ and $x^{1-\alpha}$ by functions $f$ and $g$ under certain conditions. These allow us to define the generalized covariance, the generalized variance, the generalized correlation and the generalized Wigner--Yanase--Dyson skew information related to the tracial positive maps and functions $f$ and $g$. \\
The rest of the paper is organized as follows. In the next section, we provide some preliminaries and background material. In Section \ref{sec3}, we use some techniques in the noncommutative setting to give some Cauchy--Schwarz type inequalities for the generalized covariance and the generalized variance. Then we use them to extend inequalities (\ref{heis1}) and (\ref{schr}) for tracial positive linear maps between $C^*$-algebras. In Section \ref{sec4}, we present some inequalities and properties for the generalized correlation and the generalized Wigner--Yanase--Dyson skew information. In this section, we give a generalization of inequality (\ref{semic}) for tracial positive linear maps between $C^*$-algebras. Finally, in Section \ref{sec5}, we establish some inequalities between variance and Wigner--Yanase--Dyson skew information. We indeed apply some arguments differing from the classical theory to investigate inequalities related to the generalized covariance, the generalized variance, the generalized correlation, and the generalized Wigner--Yanase--Dyson skew information.
\section{Preliminaries}
Let us fix our notation and terminology used throughout the paper. Let $\mathbb{B}(\mathcal{H})$ stand for  the $C^*$-algebra of all bounded linear operators on a complex Hilbert space $(\mathcal H, \langle\cdot,\cdot\rangle)$ with the unit $I$. An operator $A$ is called positive if $\langle Ax, x\rangle\geq 0 $ for all $x \in \mathcal{H}$, and we then write $A\geq 0$. An operator A is said to be strictly positive (denoted by $ A >0$) if it is a positive invertible operator.  Let $\leq$ be the L\"{o}wner order on the self-adjoint part of B(H). In the case that  $\mathcal{H}=\mathbb{C}^n$,  $\mathbb{B}(\mathbb{C}^n)$ is the same as  the matrix algebra $M_n(\mathbb{C})$ consist of all $n\times n$ complex matrices. Due to the Gelfand--Naimark--Segal theorem, every $C^*$-algebra can be regarded as a $C^*$-subalgebra of $\mathbb{B}(\mathcal{H})$ for some Hilbert space $\mathcal{H}$. We use $\mathcal{A},\mathcal{B}, \dots $ to denote $C^*$-algebras. We denote by ${\rm Re}(A)$ and ${\rm Im}(A)$ the real and imaginary parts of $A$, respectively. The self-adjoint part  of $\mathcal{A}$ is denoted by $\mathcal{A}_h$.
A linear map $\Phi :\mathcal{A}\to \mathcal{B}$ between $C^*$-algebras is said to be $*$-linear, if $\Phi(A^*)=\Phi(A)^*$. It is positive, if $\Phi(A)\geq 0$ whenever $A \geq 0$.
We say that $\Phi$ is unital if $\mathcal{A}$ and $ \mathcal{B}$ are unital and $\Phi$ preserves the unit. A linear map $\Phi$ is called $n$-positive if the induced map $\Phi_n : M_n(\mathcal{A})\to M_n(\mathcal{B})$ given by $\Phi_n\left([a_{ij}]\right) = \left[\Phi(a_{ij})\right]$ is positive, where $M_n(\mathcal{A})$ is the $C^*$-algebra of
$n \times n$ matrices with entries in $\mathcal{A}$. If  $\Phi$ is $n$-positive for every $n\in \mathbb{N}$, then $\Phi$ is called completely positive. It is known that if the range of a positive linear map $\Phi$  is commutative, then $\Phi$ is completely positive; see \cite[Theorem 1.2.4]{stormer}.\\
A map $\Phi$ is called tracial if $\Phi(AB)=\Phi(BA)$ for all $A$ and $B$ in the domain of $\Phi$. The usual trace on the trace class operators acting on a Hilbert space is a tracial positive linear functional. It is known that every tracial positive map between $C^*$-algebras is completely positive; see \cite [page 57]{choi}. For a given closed two sided ideal $\mathcal{I}$ of a $C^*$-algebra $\mathcal{A}$, the commutativity of the quotient $\mathcal{A}/\mathcal{I}$  is equivalent to the existence of a tracial positive linear map $\Phi: \mathcal{A}\to \mathcal{A}$ satisfying $\Phi(\Phi(A))=\Phi(A)$ and $\Phi(A)-A\in\mathcal{I}$; see \cite{choi} for more examples and implications of the definition. For a tracial positive linear map $\Phi$, a positive element $\rho \in \mathcal{A}$ is said to be a $\Phi$-density operator if $\Phi(\rho)=I$. A unital $C^*$-algebra $\mathcal{B}$ is said to be injective whenever for every unital $C^*$-algebra $\mathcal{A}$ and for every self-adjoint subspace $S$ of $\mathcal{A}$, each unital completely positive linear map from $S$ into $\mathcal{B}$, can be extended to a completely positive linear map from $\mathcal{A}$ into $\mathcal{B}$.\\
A pair $(f,g)$ of continuous real-valued functions defined on a set $D$ is called same monotonic if $(f(x)-f(y)) (g(x)-g(y)) \geq 0$ for every $x,y \in D$. Bourin \cite{bur1} and Fujii \cite{jfuji} showed that ${\rm tr}(f(\rho) g(\rho) A^2) - {\rm tr}(f(\rho) A g(\rho)A) \geq 0$ for self-adjoint matrices $A$ and $\rho$ and for all continuous real-valued functions $f$ and $ g$ on the spectrum of $\rho$ with the same monotonically. \\
If $\mathcal{B}$ is a $C^*$-subalgebra of $\mathcal{A}$, then a conditional expectation $\mathcal{E} : \mathcal{A}\to \mathcal{B}$ is a contractive positive linear map such that $\mathcal{E}(BAC)=B\mathcal{E}(A)C$ for every $A\in\mathcal{A}$ and all $B,C \in \mathcal{B}$. 

Our investigation is based on the following definition.
\begin{definition}\label{def}
Let $\Phi$ be a tracial positive linear map from a $C^*$-algebra $\mathcal{A}$ into a unital  $C^*$-algebra $\mathcal{B}$, and let $\rho\in \mathcal{A}_h$. Then for  a continuous positive real-valued function $f$ and a   continuous real-valued  function $g$  which are defined on an interval containing the spectrum of $\rho$ with $\Phi(f(\rho)g(\rho)^2)>0$,
\begin{align*}
{\rm Cov}_{\rho,\Phi}^{f,g}(A,B)&:= \Phi\left(f(\rho) A^*B \right)-\Phi\left(f(\rho)g(\rho)A^*\right) \Phi\left(f(\rho)g(\rho)^2\right)^{-1}\Phi\left(f(\rho) g(\rho) B\right) \\ {\rm Var}_{\rho,\Phi}^{f,g}(A)&:={\rm Cov}_{\rho,\Phi}^{f,g}(A,A),
\end{align*}
are called the generalized covariance and the generalized variance of $A$ and $B$, respectively. In addition, for continuous real-valued functions $f$ and $g$, which are defined on an interval containing the spectrum of $\rho$, the generalized correlation and the generalized Wigner--Yanase--Dyson skew information of two elements $A$ and $B$ are defined by
\begin{align*}
{\rm Corr}_{\rho,\Phi}^{f,g}(A,B)&:= \Phi\left(f(\rho) g(\rho) A^*B \right)-\Phi\left(f(\rho) A^* g(\rho) B \right) \\ {\rm I}_{\rho,\Phi}^{f,g}(A)&:={\rm Corr}_{\rho,\Phi}^{f,g}(A,A)\,,
\end{align*}
respectively.
\end{definition}
If we consider $f(x)=x$, $g(x)=1$  and $\rho$ as a density operator, then ${\rm Cov}_{\rho, tr}^{f,g}$ and ${\rm Var}_{\rho, tr}^{f,g}$ are the same classical covariance and variance, respectively. Moreover, for $f(x)=x^{1-\alpha}$ and $g(x)=x^{\alpha}$, ${\rm Corr}_{\rho, tr}^{f,g}$ and ${\rm I}_{\rho, tr}^{f,g}$ are the one-parameter correlation and the one-parameter Wigner--Yanase--Dyson skew information, respectively.\\
In the case when $f(x)=x^{1-\alpha}$ and $g(x)=x^{\alpha}$, we simply denote ${\rm Corr}_{\rho,\Phi}^{f,g}$, and ${\rm I}_{\rho,\Phi}^{f,g}$ by ${\rm Corr}_{\rho,\Phi}^{\alpha}$, and ${\rm I}_{\rho,\Phi}^{\alpha}$, respectively.
It is known that for every tracial positive linear map, the matrix
\begin{eqnarray*}
 \begin{bmatrix} {\rm Var}_{\rho,\Phi}(A) & {\rm Cov}_{\rho,\Phi}(A,B)\\ {\rm Cov}_{\rho,\Phi}(B,A) & {\rm Var}_{\rho,\Phi}(B)\end{bmatrix}
\end{eqnarray*}
is positive, which is equivalent to
\begin{eqnarray}\label{varcovi}
{\rm Var}_{\rho,\Phi}(A)\geq {\rm Cov}_{\rho,\Phi}(A,B) \left({\rm Var}_{\rho,\Phi}(B)\right)^{-1} {\rm Cov}_{\rho,\Phi}(B,A),
\end{eqnarray}
and is called the variance-covariance inequality; see \cite{aram}.\\

To achieve our results we need the following known lemma. The reader may consult the survey paper \cite{MKX}.

\begin{lemma}\label{matrixp} \cite[Lemma 2.1]{choi}
Let $A\geq 0, B> 0$ be two operators in $ M_n(\mathcal{A})$. Then the block matrix $\begin{bmatrix}A& X\\ X^*& B \end{bmatrix}$ is positive if and only if $A\geq X B^{-1} X^*$.
\end{lemma}


\section{Cauchy--Schwarz type inequalities related to uncertainty relation}\label{sec3}
In this section, we give some Cauchy--Schwarz type inequalities for the generalized covariance and the generalized variance. The results of this section are generalizations of Heisenberg's uncertainty relation; see \cite{fujii1}.\\
First, we state the variance-covariance inequality for the generalized covariance and variance. Its proof involves some standard matrix tricks but we prove it for the sake of convenience.

\begin{lemma}\label{varcovg}
For every tracial positive linear map $\Phi$, the matrix
\begin{eqnarray*}
\begin{bmatrix} {\rm Var}_{\rho,\Phi}^{f,g}(A) & {\rm Cov}_{\rho,\Phi}^{f,g}(A,B)\\ {\rm Cov}_{\rho,\Phi}^{f,g}(B,A) & {\rm Var}_{\rho,\Phi}^{f,g}(B)\end{bmatrix}
\end{eqnarray*}
is positive.
\end{lemma}

\begin{proof}
We have
\begin{eqnarray*}
 0&\leq&\begin{bmatrix} f(\rho)^{\frac{1}{2}}A^*&0&0 \\ f(\rho)^{\frac{1}{2}}B^*&0&0 \\ f(\rho)^{\frac{1}{2}} g(\rho) &0&0 \end{bmatrix}\begin{bmatrix} A f(\rho)^{\frac{1}{2}}&B f(\rho)^{\frac{1}{2}}& g(\rho) f(\rho)^{\frac{1}{2}}\\ 0&0&0 \\ 0&0&0 \end{bmatrix}\\ &=&\begin{bmatrix} f(\rho)^{\frac{1}{2}} A^*Af(\rho)^{\frac{1}{2}} & f(\rho)^{\frac{1}{2}}A^*B f(\rho)^{\frac{1}{2}}& f(\rho)^{\frac{1}{2}} A^* g(\rho) f(\rho)^{\frac{1}{2}} \\ f(\rho)^{\frac{1}{2}}B^*A f(\rho)^{\frac{1}{2}} & f(\rho)^{\frac{1}{2}}B^*B f(\rho)^{\frac{1}{2}} & f(\rho)^{\frac{1}{2}}B^* g(\rho) f(\rho)^{\frac{1}{2}}\\ f(\rho)^{\frac{1}{2}}g(\rho)A f(\rho)^{\frac{1}{2}} & f(\rho)^{\frac{1}{2}}g(\rho) B f(\rho)^{\frac{1}{2}}& f(\rho)^{\frac{1}{2}}g(\rho)^2 f(\rho)^{\frac{1}{2}}
\end{bmatrix}.
\end{eqnarray*}
It follows from the three-positivity and the tracial property of $\Phi$ that
 \begin{eqnarray*}
\begin{bmatrix} \Phi(f(\rho) A^*A) & \Phi(f(\rho) A^*B) & \Phi(f(\rho) g(\rho) A^*) \\ \Phi(f(\rho) B^*A) &\Phi(f(\rho) B^*B)& \Phi(f(\rho) g(\rho) B^*) \\ \Phi(f(\rho)g(\rho) A)& \Phi(f(\rho) g(\rho) B) & \Phi\left(f(\rho) g(\rho)^2\right)
\end{bmatrix} \geq 0.
\end{eqnarray*}
The positivity of the above matrix implies that
 \[
 \begin{bmatrix} \Phi(f(\rho) A^*A) & \Phi(f(\rho) A^*B) & \Phi(f(\rho) g(\rho) A^*) &0 \\ \Phi(f(\rho) B^*A) &\Phi(f(\rho) B^*B)& \Phi(f(\rho) g(\rho) B^*) &0  \\ \Phi(f(\rho)g(\rho) A)& \Phi(f(\rho) g(\rho) B) & \Phi\left(f(\rho) g(\rho)^2\right)&0 \\ 0 & 0 & 0 &  \Phi\left(f(\rho) g(\rho)^2\right)
\end{bmatrix} \geq 0.
 \] 
Hence, by employing Lemma \ref{matrixp}, we get
{\footnotesize \begin{align*}
  \begin{bmatrix} \Phi(f(\rho) g(\rho) A)^* & 0 \\ \Phi(f(\rho) g(\rho) B)^* & 0\end{bmatrix}\begin{bmatrix} \Phi\left(f(\rho)g(\rho) g(\rho)\right)^{-1} & 0\\ 0 &\Phi\left(f(\rho)g(\rho) g(\rho)\right)^{-1}
\end{bmatrix} &  \begin{bmatrix} \Phi(f(\rho) g(\rho)  A) & \Phi(f(\rho) g(\rho) B) \\ 0 & 0\end{bmatrix} \\ &\leq \begin{bmatrix} \Phi(f(\rho) A^*A) & \Phi(f(\rho) A^*B) \\ \Phi(f(\rho) B^*A) & \Phi(f(\rho) B^*B)
\end{bmatrix},
\end{align*}}
whence
{\footnotesize\begin{align*}
& \begin{bmatrix} \Phi(f(\rho) g(\rho) A)^* \Phi\left(f(\rho)g(\rho)^2\right)^{-1}\Phi(f(\rho) g(\rho) A) & \Phi(f(\rho) g(\rho) A)^* \Phi\left(f(\rho)g(\rho)^2  \right)^{-1}\Phi(f(\rho)g(\rho)  B) \\ \Phi(f(\rho) g(\rho) B)^* \Phi\left(f(\rho)g(\rho)^2 \right)^{-1}\Phi(f(\rho) g(\rho) A) & \Phi(f(\rho)g(\rho) B)^* \Phi\left(f(\rho)g(\rho)^2\right)^{-1}\Phi(f(\rho)g(\rho)  B)\end{bmatrix} \\ & \quad \leq \begin{bmatrix} \Phi(f(\rho) A^*A) & \Phi(f(\rho) A^*B) \\ \Phi(f(\rho) B^*A) & \Phi(f(\rho) B^*B)
\end{bmatrix},
\end{align*}}
or equivalently,
\begin{eqnarray*}
 \begin{bmatrix} {\rm Var}_{\rho,\Phi}^{f,g}(A) & {\rm Cov}_{\rho,\Phi}^{f,g}(A,B)\\ {\rm Cov}_{\rho,\Phi}^{f,g}(B,A) & {\rm Var}_{\rho,\Phi}^{f,g}(B)\end{bmatrix}\geq 0.
 \end{eqnarray*}
\end{proof}

Now, we are ready to give a generalization for Schr\"{o}dinger's uncertainty relation for a tracial positive linear map.

\begin{proposition}\label{gheiscom}
Let $\Phi$ be a tracial positive linear map from a $C^*$-algebra $\mathcal{A}$ into a unital  $C^*$-algebra $\mathcal{B}$, and let $\rho \in \mathcal{A}_h$. If $\Phi(\mathcal{A})$ is a commutative subset of $\mathcal{B}$, then the matrices
\begin{eqnarray*}
\begin{bmatrix} {\rm Var}_{\rho,\Phi}^{f,g}(A) & {\rm Re} \left({\rm Cov}_{\rho,\Phi}^{f,g}(A,B)\right)+\frac{1}{2} \Phi(f(\rho) [A,B])\\ {\rm Re} \left({\rm Cov}_{\rho,\Phi}^{f,g}(A,B)\right)-\frac{1}{2}\Phi(f(\rho) [A,B])& {\rm Var}_{\rho,\Phi}^{f,g}(B)
\end{bmatrix}
\end{eqnarray*}
and
\begin{eqnarray*}
\begin{bmatrix} {\rm Var}_{\rho,\Phi}^{f,g}(A) & \frac{1}{2} \Phi(f(\rho) [A,B])\\ -\frac{1}{2}\Phi(f(\rho) [A,B])& {\rm Var}_{\rho,\Phi}^{f,g}(B)
\end{bmatrix}
\end{eqnarray*}
are positive for all $A,B \in \mathcal{A}_h$, and all continuous positive real-valued functions $f$ and $g$  on the spectrum $\rho$ with $\Phi(f(\rho)g(\rho)^2)>0$.
\end{proposition}

\begin{proof}
We have 
{\footnotesize\begin{align*}
{\rm Cov}_{\rho,\Phi}^{f,g}(A,B)-{\rm Cov}_{\rho,\Phi}^{f,g}(B,A)&= \Phi\left(f(\rho) AB\right)- \Phi\left(f(\rho)g(\rho)  A)\Phi(f(\rho)g(\rho)^2\right)^{-1}\Phi\left(f(\rho)g(\rho)  B\right)\\ 
&	\ \ - \Phi\left(f(\rho) BA\right)+\Phi\left(f(\rho)g(\rho)  B\right)\Phi\left(f(\rho)g(\rho)^2 \right)^{-1}\Phi\left(f(\rho)g(\rho)  A\right)\\ &=\Phi\left(f(\rho)[A,B]\right) \tag{since $\Phi(\mathcal{A})$ is commutative}
\end{align*}}
and an easy calculation shows that
\begin{eqnarray*}
{\rm Cov}_{\rho,\Phi}^{f,g}(A,B)+{\rm Cov}_{\rho,\Phi}^{f,g}(B,A) =2{\rm Re}\left({\rm Cov}_{\rho,\Phi}^{f,g}(A,B)\right).
\end{eqnarray*}
Summing both sides of the above equalities, we get
\begin{eqnarray*}
2{\rm Cov}_{\rho,\Phi}^{f,g}(A,B) =\Phi\left(f(\rho)[A,B]\right)+2{\rm Re}\left({\rm Cov}_{\rho,\Phi}^{f,g}(A,B)\right).
\end{eqnarray*}
Since $\Phi(f(\rho)[A,B])^* =-\Phi(f(\rho)[A,B])$ and the range of $\Phi$ is commutative, we get
\begin{eqnarray*}\label{rabete}
\left|{\rm Cov}_{\rho,\Phi}^{f,g}(A,B)\right|^2=\left|{\rm Re} \left({\rm Cov}_{\rho,\Phi}^{f,g}(A,B)\right)\right|^2+ \frac{1}{4} \left|\Phi(f(\rho) [A,B])\right|^2.
\end{eqnarray*}
By a continuity argument we can assume that ${\rm Var}_{\rho,\Phi}^{f,g}(B)>0$ and by using Lemma \ref{varcovg}, we conclude that
\begin{eqnarray*}
{\rm Var}_{\rho,\Phi}^{f,g}(A){\rm Var}_{\rho,\Phi}^{f,g}(B)\geq \left|{\rm Re} \left({\rm Cov}_{\rho,\Phi}^{f,g}(A,B)\right)\right|^2+ \frac{1}{4} \left|\Phi(f(\rho) [A,B])\right|^2.
\end{eqnarray*}
Hence,

\begin{align*}
{\rm Var}_{\rho,\Phi}^{f,g}(A) &=\left(\left|{\rm Re} \left({\rm Cov}_{\rho,\Phi}^{f,g}(A,B)\right)\right|^2 + \frac{1}{4} \left|\Phi\left(f(\rho) [A,B]\right)\right|^2\right) \left({\rm Var}_{\rho,\Phi}^{f,g}(B)\right)^{-1}\\ &=\left({\rm Re} \left({\rm Cov}_{\rho,\Phi}^{f,g}(A,B)\right)^2 -\frac{1}{4} \Phi(f(\rho) [A,B])^2\right) \left({\rm Var}_{\rho,\Phi}^{f,g}(B)\right)^{-1} \\ & \tag{since $\Phi(f(\rho)[A,B])^* =-\Phi(f(\rho)[A,B]$}
\\ &= \left({\rm Re} \left({\rm Cov}_{\rho,\Phi}^{f,g}(A,B)\right) + \frac{1}{2} \Phi\left(f(\rho) [A,B]\right) \right) \\ &\quad \cdot \left({\rm Var}_{\rho,\Phi}^{f,g}(B)\right)^{-1} \left (  {\rm Re} \left({\rm Cov}_{\rho,\Phi}^{f,g}(A,B)\right)+\frac{1}{2} \Phi(f(\rho) [A,B]) \right)^*.\\ & \tag{since the range of $\Phi$ is commutative}
\end{align*}
Now, Lemma \ref{matrixp} implies the positivity of the matrix
\begin{eqnarray*}
\begin{bmatrix} {\rm Var}_{\rho,\Phi}^{f,g}(A) & {\rm Re} \left({\rm Cov}_{\rho,\Phi}^{f,g}(A,B)\right)+\frac{1}{2} \Phi(f(\rho) [A,B])\\ {\rm Re} \left({\rm Cov}_{\rho,\Phi}^{f,g}(A,B)\right)-\frac{1}{2}\Phi(f(\rho) [A,B])& {\rm Var}_{\rho,\Phi}^{f,g}(B)
\end{bmatrix}.
\end{eqnarray*}
To prove the positivity of the second matrix, first note that the positivity of the matrix $\begin{bmatrix} {\rm Var}_{\rho,\Phi}^{f,g}(A) & {\rm Cov}_{\rho,\Phi}^{f,g}(B,A)\\ {\rm Cov}_{\rho,\Phi}^{f,g}(A,B) & {\rm Var}_{\rho,\Phi}^{f,g}(B)\end{bmatrix}
$ implies the positivity of \begin{eqnarray*} X=\begin{bmatrix} {\rm Var}_{\rho,\Phi}^{f,g}(A) & -{\rm Cov}_{\rho,\Phi}^{f,g}(B,A)\\ -{\rm Cov}_{\rho,\Phi}^{f,g}(A,B) & {\rm Var}_{\rho,\Phi}^{f,g}(B)\end{bmatrix}.
\end{eqnarray*} In addition, it follows from the commutativity of the range of $\Phi$ that
\begin{eqnarray*}
Y= \begin{bmatrix} {\rm Var}_{\rho,\Phi}^{f,g}(A) & {\rm Cov}_{\rho,\Phi}^{f,g}(A,B)\\ {\rm Cov}_{\rho,\Phi}^{f,g}(B,A) & {\rm Var}_{\rho,\Phi}^{f,g}(B)\end{bmatrix}\geq 0.
\end{eqnarray*}
Therefore,
\begin{eqnarray*}
0 \leq X+Y=\begin{bmatrix} 2{\rm Var}_{\rho,\Phi}^{f,g}(A) & \Phi\left(f(\rho)[A,B]\right)\\ -\Phi\left(f(\rho)[A,B]\right) &2 {\rm Var}_{\rho,\Phi}^{f,g}(B)\end{bmatrix},
\end{eqnarray*}
whence we arrive at the second inequality.
\end{proof}

Next, we aim to give a generalization of Heisenberg's uncertainty relation for a tracial positive linear map between $C^*$-algebra. To get this result we need some lemmas.

\begin{lemma}[Choi--Tsui]\label{choi}\cite [pages 59 -- 60]{choi}
Let $\mathcal{A},\mathcal{B}$ be $C^*$-algebras such that either one of them is $W^*$-algebra or $\mathcal{B}$ is an injective $C^*$-algebra. Let $\Phi:\mathcal{A}\longrightarrow \mathcal{B}$ be a tracial positive linear map. Then there exist a
commutative $C^*$-algebra $C(X)$ for some compact Hausdorff space $X$, tracial positive linear maps $\phi_1 : A\longrightarrow C(X)$, and $\phi_2: C(X)\longrightarrow \mathcal{B}$ such that $\Phi=\phi_2\circ \phi_1$. Moreover, in case that $\Phi$ is unital, then
 $\phi_1$ and $\phi_2$ can be chosen to be unital. In particular, $\Phi$ is completely positive.
\end{lemma}

The next lemma is a consequence of the positivity of the matrix $\begin{bmatrix}
A^*B^{-1}A & A^* \\ A & B
\end{bmatrix}$ and two-positivity of $\Phi$.
\begin{lemma}\label{kadison} Let $\Phi$ be a tracial positive linear map from a $C^*$-algebra $\mathcal{A}$ into a unital  $C^*$-algebra $\mathcal{B}$. Then
\begin{eqnarray}\label{kadi2}
\Phi(A^*B^{-1}A)\geq \Phi(A)^*\left(\Phi(B)\right)^{-1}\Phi(A)
\end{eqnarray}
for all $B>0$ and all $A\in\mathcal{A}$.
\end{lemma}

The following theorem gives a generalization of Heisenberg's uncertainty relation for tracial positive linear maps between $C^*$-algebras.
\begin{theorem}\label{heisnoncom}
Let $\mathcal{A}$ be a $C^*$-algebra  and  $\mathcal{B}$ be a unital $C^*$-algebra such that either one of them is $W^*$-algebra or $\mathcal{B}$ is an injective $C^*$-algebra. If $\Phi :\mathcal{A} \longrightarrow \mathcal{B}$ is a tracial positive linear map and $\rho \in \mathcal{A}_h$, then the matrix
\begin{eqnarray*}
\begin{bmatrix} {\rm Var}_{\rho,\Phi}^{f,g}(A) & \frac{1}{2} \Phi\left(f(\rho) [A,B]\right)\\ -\frac{1}{2}\Phi\left(h(\rho\right) [A,B])& {\rm Var}_{\rho,\Phi}^{f,g}(B)
\end{bmatrix}
\end{eqnarray*}
is positive for all $A,B \in \mathcal{A}_h$, and all continuous positive real-valued functions $f$ and $g$ on the spectrum $\rho$ with $\Phi(f(\rho)g(\rho)^2)>0$.
\end{theorem}
\begin{proof}
According to Lemma \ref{choi} there exist a commutative $C^*$-algebra $C(X)$ and tracial positive linear maps $\phi_1 : A\longrightarrow C(X)$ and $\phi_2: C(X)\longrightarrow \mathcal{B}$ such that $\Phi=\phi_2\circ \phi_1$. Hence, Proposition \ref{gheiscom} ensures the positivity of the matrix
\begin{eqnarray*}
\begin{bmatrix} {\rm Var}_{\rho,\phi_1}^{f,g}(A) & \frac{1}{2} \phi_1\left(f(\rho) [A,B]\right)\\ -\frac{1}{2}\phi_1\left(f(\rho) [A,B]\right)& {\rm Var}_{\rho,\phi_1}^{f,g}(B)
\end{bmatrix}.
\end{eqnarray*}
Since $\phi_2$ is two-positive, we get
\begin{eqnarray*}
\begin{bmatrix} \phi_2\left({\rm Var}_{\rho,\phi_1}^{f,g}(A)\right) &\frac{1}{2} \phi_2\left(\phi_1(f(\rho) [A,B])\right)\\-\frac{1}{2} \phi_2\left(\phi_1(f(\rho) [A,B])\right)& \phi_2\left({\rm Var}_{\rho,\phi_1}^{f,g}(B)\right)
\end{bmatrix}\geq 0.
\end{eqnarray*}
 In addition,
{\footnotesize\begin{align*}
 \phi_2\left({\rm Var}_{\rho,\phi_1}^{f,g}(A)\right)&= \phi_2\left(\phi_1(f(\rho) A^2)-\phi_1(f(\rho)g(\rho) A) \left(\phi_1(f(\rho)g(\rho)^2 )\right)^{-1}\phi_1(f(\rho) g(\rho) A)\right)\\ 
 &=\Phi\left(f(\rho) A^2\right) - \phi_2\left(\phi_1(f(\rho)g(\rho) A) \left(\phi_1\left(f(\rho)g(\rho)^2\right)\right)^{-1}\phi_1(f(\rho) g(\rho) A)\right)\\ &\leq \Phi\left(f(\rho) A^2\right)\\ & \quad -(\phi_2\circ\phi_1)\left(f(\rho)g(\rho) A\right) \left(\left(\phi_2\circ\phi_1)(f(\rho)g(\rho)^2 \right)\right)^{-1}(\phi_2\circ\phi_1)(f(\rho)g(\rho)  A)\\ & \tag{by inequality \eqref{kadi2}}\\
 &= \Phi\left(f(\rho) A^2\right) -\Phi(f(\rho)g(\rho)A) \left(\Phi\left(f(\rho)g(\rho)^2\right)\right)^{-1}\Phi(f(\rho) g(\rho)A)\\
 &={\rm Var}_{\rho,\Phi}^{f,g}(A).
\end{align*}}
Similarly, we can get $ {\rm Var}_{\rho,\Phi}^{f,g}(B)\geq \phi_2\left({\rm Var}_{\rho,\phi_1}^{f,g}(B)\right)$. Hence,
\begin{eqnarray*}
0 &\leq & \begin{bmatrix} \phi_2\left({\rm Var}_{\rho,\phi_1}^{f,g}(A)\right) &\frac{1}{2} \phi_2\left(\phi_1(f(\rho) [A,B])\right)\\-\frac{1}{2} \phi_2\left(\phi_1(f(\rho) [A,B])\right)& \phi_2\left({\rm Var}_{\rho,\phi_1}^{f,g}(B)\right)
\end{bmatrix}
\\ &\leq& \begin{bmatrix} {\rm Var}_{\rho,\Phi}^{f,g}(A) & \frac{1}{2} \Phi\left(f(\rho) [A,B]\right)\\ -\frac{1}{2}\Phi\left(h(\rho\right) [A,B])& {\rm Var}_{\rho,\Phi}^{f,g}(B)
\end{bmatrix}.
\end{eqnarray*}
\end{proof}


 \section{Some properties of correlation}\label{sec4}

We intend to investigate the positivity of the generalized Wigner--Yanase--Dyson skew information. We extend the results of
 \cite{bur1} and \cite{jfuji} to the positivity of the generalized Wigner--Yanase--Dyson skew information.\\
It is easy to see that the same monotonicity of a pair $(f,g)$ of continuous real-valued functions defined on a set $D$ is equal to the validity of the inequality $f(x)g(x)+f(y)g(y) \geq f(x)g(y)+f(y)g(x)$ for every $x,y \in D$.
\begin{theorem}\label{posiwigner1}
Let $\Phi:\mathcal{A}\longrightarrow \mathcal{B}$ be tracial positive linear map between von Neumann algebras. Then
\begin{eqnarray}\label{posiwigner}
\Phi(f(\rho) g(\rho)A^2) \geq \Phi(f(\rho)A g(\rho)A)
\end{eqnarray}
for each pair of same monotonic functions $f$ and $g$ defined on the spectrum of $\rho\in \mathcal{A}_h $ and each element $A \in \mathcal{A}_h$. In particular, 
$
{\rm I}_{\rho,\Phi}^{f,g}(A) \geq 0.
$
\end{theorem}
\begin{proof}
Since $\rho$ is self-adjoint, there exists a sequence of self-adjoint operators converging to $\rho$ such that each term $P$ of the sequence has the spectral representation $P=\sum_{i=1}^n \lambda_i E_i$, whenever $\lambda_i$, $ i=1,\dots , n$ are real numbers and $E_i$'s are mutually orthogonal projections with $ \sum_{i=1}^n E_i =I$. Hence, we only need to prove inequality \eqref{posiwigner} for a such self-adjoint operator $P=\sum_{i=1}^n \lambda_i E_i$. \\
First note that $f(P)g(P)A^2=\sum_{i=1}^n f(\lambda_i)g(\lambda_i) E_iA^2$. In addition,
\begin{align*}
f(P)Ag(P)A&=\sum_{i=1}^n f(\lambda_i) E_iA \sum_{j=1}^n g(\lambda_j) E_jA
=\sum_{i=1}^n f(\lambda_i)g(\lambda_i)E_iAE_iA \\ &\quad +\sum_{i< j}\left(f(\lambda_i)g(\lambda_j)E_iAE_jA +f(\lambda_j)g(\lambda_i)E_jAE_iA\right).
\end{align*}
 Since $ \sum_{i=1}^n E_i =I$, we can write
$
E_iA^2=E_iAE_1A+E_iAE_2A+\cdots+ E_iAE_nA.
$
Consequently,
{\footnotesize\begin{align*}
f(\lambda_i)g(\lambda_i)E_iA^2&=f(\lambda_i)g(\lambda_i)E_iAE_1A+\cdots+f(\lambda_i)g(\lambda_i) E_iAE_nA.
 \end{align*}
Hence,
\begin{align*}
&\Phi \left(f(P)g(P)A^2\right)\\
&=\Phi\left(\sum_{i=1}^n f(\lambda_i)g(\lambda_i)E_iA^2\right) \\
&=\Phi\left(\sum_{i=1}^n f(\lambda_i)g(\lambda_i)E_iAE_iA\right) +\Phi \left(\sum_{i< j}(f(\lambda_i)g(\lambda_i)E_iAE_jA +f(\lambda_j)g(\lambda_j)E_jAE_iA) \right)
\\
&=\Phi\left(\sum_{i=1}^n f(\lambda_i)g(\lambda_i)E_iAE_iA\right) +\Phi \left(\sum_{i< j}(f(\lambda_i)g(\lambda_i)E_iAE_jA +f(\lambda_j)g(\lambda_j)E_iAE_jA) \right) \\ & \tag{since $\Phi$ is tracial} 
\\ &\quad \geq
\Phi\left(\sum_{i=1}^n f(\lambda_i)g(\lambda_i)E_iAE_iA\right) +\Phi \left(\sum_{i< j}(f(\lambda_i)g(\lambda_j)E_iAE_jA +f(\lambda_j)g(\lambda_i)E_iAE_jA) \right) \\
& \tag{since $f,g$ are same monotonic} \\
 &= \Phi\left(\sum_{i=1}^n f(\lambda_i)E_iA\sum_{j=1}^ng(\lambda_j) E_jA \right)\\
 &=\Phi \left(f(P)Ag(P)A\right).
\end{align*}}
\end{proof}

\begin{theorem}\label{reposwing}
Let $\Phi: \mathcal{A}\to \mathcal{B}$ be a tracial positive linear map between von Neumann algebras. Then
 \begin{eqnarray*}
{\rm I}_{\rho,\Phi}^{f,g}(A)+{\rm I}_{\rho,\Phi}^{f,g}(A^*) \geq 0.
\end{eqnarray*}
for each pair of same monotonic functions $f$ and $g$ defined on the spectrum of a self-adjoint element $\rho$ and each operator $A \in \mathcal{A}$.
\end{theorem}

\begin{proof}
Define the map $\Psi : M_2(\mathcal{A}) \to M_2(\mathcal{B})$ by
\begin{eqnarray}
\Psi\left( \begin{bmatrix}
A &B \\ C & D
\end{bmatrix}\right) = \begin{bmatrix}
\dfrac{\Phi(A) +\Phi(D)}{2}&0 \\ 0 & 0
\end{bmatrix}.
\end{eqnarray}

 Clearly, $\Psi$ is a tracial positive linear map. Let $A\in \mathcal{A}$. Then the matrices $\begin{bmatrix} 0 & A \\ A^* & 0 \end{bmatrix} $ and $\begin{bmatrix} \rho & 0 \\ 0 & \rho \end{bmatrix}$ are self-adjoint elements of $M_2(\mathcal{A})$. Furthermore, the spectrums of $\rho$ and $
 \begin{bmatrix}
 \rho & 0 \\ 0 & \rho
 \end{bmatrix}
 $ are equal. Hence,
\begin{align*}
&\hspace{-1cm} \begin{bmatrix}
 \Phi(f(\rho)g(\rho)A^*A)+\Phi(f(\rho)g(\rho)AA^*) & 0 \\ 0 & 0
 \end{bmatrix}\\
\\ &\ = 2\Psi \left(\begin{bmatrix}
 f(\rho)g(\rho) & 0 \\ 0 & f(\rho)g(\rho)
 \end{bmatrix} \begin{bmatrix}
 0 & A \\ A^* & 0
 \end{bmatrix} \begin{bmatrix}
 0 & A \\ A^* & 0
 \end{bmatrix}\right)
\\ \\ &\ \geq
2 \Psi \left(\begin{bmatrix}
 f(\rho) & 0 \\ 0 & f(\rho)
 \end{bmatrix} \begin{bmatrix}
 0 & A \\ A^* & 0
 \end{bmatrix} \begin{bmatrix}
 g(\rho) & 0 \\ 0 & g(\rho)
 \end{bmatrix}\begin{bmatrix}
 0 & A \\ A^* & 0
 \end{bmatrix}\right)
 \\ \tag{by Theorem \ref{posiwigner1}}\\ \\ &\ = \begin{bmatrix}
 \Phi(f(\rho) A^*g(\rho) A) + \Phi(f(\rho) Ag(\rho) A^*) & 0 \\ 0 & 0
 \end{bmatrix},
\end{align*}
which ensures that
\begin{eqnarray*}
{\rm I}_{\rho,\Phi}^{f,g}(A)+{\rm I}_{\rho,\Phi}^{f,g}(A^*) \geq 0.
\end{eqnarray*}
\end{proof}
\begin{definition}\label{semi}
Let $\Phi: \mathcal{A}\to \mathcal{B}$ be a tracial positive linear map, and let $\rho\in \mathcal{A}_h$. Then for elements $A,B \in \mathcal{A}$, we set
\begin{eqnarray*}
{\rm Corr}_{\rho,\Phi}^{\prime f,g}(A,B):=\dfrac{1}{2} \left({\rm Corr}_{\rho,\Phi}^{f,g}(A,B)+{\rm Corr}_{\rho,\Phi}^{{f,g}}(B^*,A^*)\right)
\end{eqnarray*}
 and ${\rm I}_{\rho,\Phi}^{\prime f,g}(A):={\rm Corr}_{\rho,\Phi}^{\prime f,g}(A,A)$, which are called the generalized symmetric correlation and the generalized symmetric Wigner--Yanase--Dyson skew information, respectively. \\
It is easy to check that ${\rm Corr}_{\rho,\Phi}^{\prime f,g}(A,B)$ has the following properties:
\begin{itemize}
\item [(i)] ${\rm Corr}_{\rho,\Phi}^{\prime f,g}(A,A) \geq 0, {\rm \ for\ every}\ A\in \mathcal{A}, \ {\rm (see\ Theorem\ \ref{reposwing})}$,
\item [(ii)] ${\rm Corr}_{\rho,\Phi}^{\prime f,g}(A,B+\lambda C) ={\rm Corr}_{\rho,\Phi}^{\prime f,g}(A,B) +\lambda {\rm Corr}_{\rho,\Phi}^{\prime f,g}(A,C), {\rm \ for\ all}\ A,B\\ {\rm in}\ \mathcal{A} {\rm \ and\ every\ } \lambda \in \mathbb{C}$,
\item [(iii)] ${\rm Corr}_{\rho,\Phi}^{\prime f,g}(A,B)^*={\rm Corr}_{\rho,\Phi}^{\prime f,g}(B,A)$.
\end{itemize}
In addition, if $A$ and $B$ are self-adjoint, then $ {\rm Corr}_{\rho,\Phi}^{\prime f,g}(A,B) ={\rm Re} \left({\rm Corr}_{\rho,\Phi}^{f,g}(A,B)\right)$ and ${\rm I}_{\rho,\Phi}^{\prime f,g}(A)={\rm I}_{\rho,\Phi}^{{f,g}}(A)$.
\end{definition}
Now we give a generalization of inequality (\ref{semic}) for a tracial positive linear map. The following theorem gives a Cauchy--Schwarz type inequality for the generalized correlation.
\begin{theorem}
Let $\Phi: \mathcal{A}\to \mathcal{B}$ be a tracial positive linear map between von Neumann algebras. Then
\[\begin{bmatrix}
{\rm I}_{\rho,\Phi}^{f,g}(A) & {\rm Corr}_{\rho,\Phi}^{\prime f,g}(A,B) \\ {\rm Corr}_{\rho,\Phi}^{\prime f,g}(B,A) & {\rm I}_{\rho,\Phi}^{f,g}(B)
\end{bmatrix}=
\begin{bmatrix}
{\rm I}_{\rho,\Phi}^{f,g}(A) & {\rm Re} \left({\rm Corr}_{\rho,\Phi}^{f,g}(A,B)\right) \\ {\rm Re} \left({\rm Corr}_{\rho,\Phi}^{f,g}(A,B)\right) & {\rm I}_{\rho,\Phi}^{f,g}(B)
\end{bmatrix}\geq 0
\]
for any pair of same monotonic functions $f$ and $g$ defined on the spectrum of $\rho$ and any $A,B \in \mathcal{A}_h$.
\end{theorem}
\begin{proof}
By a similar argument as the first paragraph of the proof of Theorem \ref{posiwigner1} we only need to prove the theorem in the case $\rho = \sum_{i=1}^n \lambda_i E_i$, where $\lambda_i$ ; $ i=1,\dots , n$ are real numbers and $E_i$'s are orthogonal projections with 
$\sum_{i=1}^n E_i =I$.  We denote $f(\lambda_i), g(\lambda_j)$ by $f_i, g_j$, respectively. And let $\delta_{ij} = f_ig_i+f_jg_j$,~~
$\xi_{ij} = f_ig_j+f_jg_i$,~~$\Delta_{ij} = \Phi(E_iAE_jB)+\Phi(E_jAE_iB)$,~~ $V = \Phi(f(\rho)Ag(\rho)B)+\Phi(f(\rho)Bg(\rho)A)$ and 
$W = \Phi(f(\rho)g(\rho)AB)+\Phi(f(\rho)g(\rho)BA)$. 
Then we have 
\begin{align*}
\Phi(f(\rho)g(\rho)A^2) &= \sum_{i=1}^n f_ig_i\Phi(E_iA^2)\\
&= \sum_{i=1}^n \sum_{j=1}^n f_ig_i\Phi(E_iAE_jA) \tag{since $\sum_{j=1}^n E_j =I$}\\
&=\sum_{i=1}^n f_ig_i\Phi(E_iAE_iA)+\sum_{i<j} f_ig_i\Phi(E_iAE_jA)+\sum_{i>j} f_ig_i\Phi(E_iAE_jA) \\
&= \sum_{i=1}^n f_ig_i\Phi(E_iAE_iA)+\sum_{i<j} f_ig_i\Phi(E_iAE_jA)+\sum_{i<j} f_jg_j\Phi(E_jAE_iA) \\
&= \sum_{i=1}^n f_ig_i\Phi(E_iAE_iA)+\sum_{i<j} (f_ig_i+f_jg_j)\Phi(E_iAE_jA) \\
&= \sum_{i=1}^n f_ig_i\Phi(E_iAE_iA)+\sum_{i<j} \delta_{ij}\Phi(E_iAE_jA). 
\end{align*}
Similarly, we have $\Phi(f(\rho)g(\rho)B^2) = \sum_{i=1}^n f_ig_i\Phi(E_iBE_iB)+\sum_{i<j} \delta_{ij}\Phi(E_iBE_jB)$. 
\begin{align*}
\Phi(f(\rho)Ag(\rho)A) &= \sum_{i=1}^n \sum_{j=1}^n f_ig_j\Phi(E_iAE_jA) \\
&= \sum_{i=1}^n f_ig_i\Phi(E_iAE_iA)+\sum_{i<j} f_ig_j\Phi(E_iAE_jA)+\sum_{i>j} f_ig_j\Phi(E_iAE_jA) \\
&= \sum_{i=1}^n f_ig_i\Phi(E_iAE_iA)+\sum_{i<j} f_ig_j\Phi(E_iAE_jA)+\sum_{i<j} f_jg_i\Phi(E_jAE_iA) \\
&= \sum_{i=1}^n f_ig_i\Phi(E_iAE_iA)+\sum_{i<j} (f_ig_j+f_jg_i)\Phi(E_iAE_jA) \\
&=\sum_{i=1}^n f_ig_i\Phi(E_iAE_iA)+\sum_{i<j} \xi_{ij}\Phi(E_iAE_jA). 
\end{align*}
Similarly, we have $\Phi(f(\rho)Bg(\rho)B) = \sum_{i=1}^n f_ig_i\Phi(E_iBE_iB)+\sum_{i<j} \xi_{ij}\Phi(E_iBE_jB)$. 
And also 
\begin{align*}
W &=\Phi(f(\rho)g(\rho)AB) + \Phi(f(\rho)g(\rho)BA) \\
&= \sum_{i=1}^n f_ig_i\Phi(E_iAB)+\sum_{i=1}^n f_ig_i\Phi(E_iBA) \\
&= \sum_{i=1}^n \sum_{j=1}^n f_ig_i\Phi(E_iAE_jB)+\sum_{i=1}^n \sum_{j=1}^n f_ig_i\Phi(E_iBE_jA) \\
&= \sum_{i=1}^n f_ig_i\Phi(E_iAE_iB)+\sum_{i<j} f_ig_i\Phi(E_iAE_jB)+\sum_{i>j} f_ig_i\Phi(E_iAE_jB) \\
& \quad + \sum_{i=1}^n f_ig_i\Phi(E_iBE_iA)+\sum_{i<j} f_ig_i\Phi(E_iBE_jA)+\sum_{i>j} f_ig_i\Phi(E_iBE_jA) \\
&= \sum_{i=1}^n f_ig_i\Phi(E_iAE_iB)+\sum_{i<j} f_ig_i\Phi(E_iAE_jB)+\sum_{i<j} f_jg_j\Phi(E_jAE_iB) \\
& \quad + \sum_{i=1}^n f_ig_i\Phi(E_iBE_iA)+\sum_{i<j} f_ig_i\Phi(E_iBE_jA)+\sum_{i<j} f_jg_j\Phi(E_jBE_iA) \\
&= \sum_{i=1}^n f_ig_i\Phi(E_iAE_iB)+\sum_{i=1}^n f_ig_i\Phi(E_iBE_iA) \\
& \quad + \sum_{i<j} (f_ig_i+f_jg_j)\Phi(E_iAE_jB)+\sum_{i<j} (f_ig_i+f_jg_j)\Phi(E_jAE_iB) \\
&=2 \sum_{i=1}^n f_ig_i\Phi(E_iAE_iB)+\sum_{i<j} \delta_{ij}\Phi(E_iAE_jB)+\sum_{i<j} \delta_{ij}\Phi(E_jAE_iB) \\
&= 2 \sum_{i=1}^n f_ig_i\Phi(E_iAE_iB)+\sum_{i<j} \delta_{ij}\Delta_{ij}.
\end{align*}
Then $\frac{1}{2}W = \sum_{i=1}^n f_ig_i\Phi(E_iAE_iB) + \frac{1}{2} \sum_{i<j} \delta_{ij}\Delta_{ij}$. Similarly, we have 
$\frac{1}{2}V = \sum_{i=1}^n f_ig_i\Phi(E_iAE_iB)+\frac{1}{2} \sum_{i<j} \xi_{ij}\Delta_{ij}$. Therefore we can write \\
{\footnotesize
\begin{align*}
& \begin{bmatrix}
\Phi(f(\rho)g(\rho)A^2 ) & \frac{1}{2}W \\ \frac{1}{2}W & \Phi(f(\rho)g(\rho)B^2)
\end{bmatrix} \\ &=\begin{bmatrix}
\sum_{i=1}^n f_ig_i \Phi(E_iAE_iA)+\sum_{i<j} \delta_{ij} \Phi(E_iAE_jA) & \sum_{i=1}^n f_ig_i \Phi(E_iAE_iB)+ \frac{1}{2} \sum_{i<j} \delta_{ij} \Delta_{ij} \\ \sum_{i=1}^n f_ig_i \Phi(E_iAE_iB) + \frac{1}{2} \sum_{i<j} \delta_{ij} \Delta_{ij} & \sum_{i=1}^n f_ig_i \Phi(E_iBE_iB) + \sum_{i<j} \delta_{ij} \Phi(E_iBE_jB)
\end{bmatrix} \\ &=\begin{bmatrix}
\sum_{i=1}^n f_ig_i \Phi(E_iAE_iA) & \sum_{i=1}^n f_ig_i \Phi(E_iAE_iB) \\
\sum_{i=1}^n f_ig_i \Phi(E_iAE_iB) & \sum_{i=1}^n f_ig_i \Phi(E_iBE_iB)
\end{bmatrix} + \begin{bmatrix}
\sum_{i<j} \delta_{ij} \Phi(E_iAE_jA) & \frac{1}{2} \sum_{i<j} \delta_{ij} \Delta_{ij} \\
\frac{1}{2} \sum_{i<j} \delta_{ij} \Delta_{ij} & \sum_{i<j} \delta_{ij} \Phi(E_iBE_jB)
\end{bmatrix} \\ &=\begin{bmatrix}
\sum_{i=1}^n f_ig_i \Phi(E_iAE_iA) & \sum_{i=1}^n f_ig_i \Phi(E_iAE_iB) \\
\sum_{i=1}^n f_ig_i \Phi(E_iAE_iB) & \sum_{i=1}^n f_ig_i \Phi(E_iBE_iB) 
\end{bmatrix} + \frac{1}{2} \sum_{i<j} \begin{bmatrix}
\delta_{ij} & 0 \\ 0 & \delta_{ij} 
\end{bmatrix} \begin{bmatrix}
2 \Phi(E_iAE_jA) & \Delta_{ij} \\ \Delta_{ij} & 2 \Phi(E_iBE_jB)
\end{bmatrix} \\ &\geq \begin{bmatrix}
\sum_{i=1}^n f_ig_i \Phi(E_iAE_iA) & \sum_{i=1}^n f_ig_i \Phi(E_iAE_iB) \\
\sum_{i=1}^n f_ig_i \Phi(E_iAE_iB) & \sum_{i=1}^n f_ig_i \Phi(E_iBE_iB)
\end{bmatrix} + \frac{1}{2} \sum_{i<j} \begin{bmatrix}
\xi_{ij} & 0 \\ 0 & \xi_{ij}
\end{bmatrix} \begin{bmatrix}
2 \Phi(E_iAE_jA) & \Delta_{ij} \\ \Delta_{ij} & 2 \Phi(E_iBE_jB)
\end{bmatrix}
\\ & \tag{since $(f,g)$ are same monotonic}
 \\ &=\begin{bmatrix}
\sum_{i=1}^n f_ig_i \Phi(E_iAE_iA) & \sum_{i=1}^n f_ig_i \Phi(E_iAE_iB) \\
\sum_{i=1}^n f_ig_i \Phi(E_iAE_iB) & \sum_{i=1}^n f_ig_i \Phi(E_iBE_iB)
\end{bmatrix} + \begin{bmatrix}
\sum_{i<j} \xi_{ij} \Phi(E_iAE_jA) & \frac{1}{2} \sum_{i<j} \xi_{ij} \Delta_{ij} \\
\frac{1}{2} \sum_{i<j} \xi_{ij} \Delta_{ij} & \sum_{i<j} \xi_{ij} \Phi(E_iBE_jB)
\end{bmatrix} \\ &=\begin{bmatrix}
\Phi(f(\rho)Ag(\rho)A) & \frac{1}{2}V \\ \frac{1}{2} V & \Phi(f(\rho)Bg(\rho)B)
\end{bmatrix}. 
\end{align*}}
Hence,
{\footnotesize
\begin{align*}
&\begin{bmatrix}
\Phi(f(\rho)g(\rho)A^2 ) & \frac{1}{2}\left(\Phi(f(\rho)g(\rho)AB) +\Phi(f(\rho)g(\rho)BA)\right) \\ \frac{1}{2}\left(\Phi(f(\rho)g(\rho)BA) +\Phi(f(\rho)g(\rho)AB)\right) & \Phi(f(\rho)g(\rho)B^2)
\end{bmatrix} \\ &\quad \geq \begin{bmatrix}\Phi(f(\rho)Ag(\rho)A ) &\frac{1}{2}\left(\Phi(f(\rho)Ag(\rho)B) +\Phi(f(\rho)Bg(\rho)A)\right) \\ \frac{1}{2}\left(\Phi(f(\rho)Bg(\rho)A) +\Phi(f(\rho)Ag(\rho)B)\right) & \Phi(f(\rho)Bg(\rho)B).
\end{bmatrix},
\end{align*}}
which yields the required result.
\end{proof}
\begin{corollary}\label{corcor}
Let $\Phi: \mathcal{A}\to \mathcal{B}$ be a tracial positive linear map between von Neumann algebras. Then
\begin{eqnarray*}
\left|{\rm Re} \left({\rm Corr}_{\rho,\Phi}^{f,g}(A,B)\right)\right|^2 \leq {\rm I}_{\rho,\Phi}^{f,g}(A) \left\|{\rm I}_{\rho,\Phi}^{f,g}(B)\right\|
\end{eqnarray*}
for any pair of same monotonic functions $f$ and $g$ defined on the spectrum of $\rho$ and any operators $A,B \in \mathcal{A}_h$.\\
In particular, if $\Phi(\mathcal{A})$ is commutative, then the above inequality can be refined to
\begin{eqnarray*}
\left|{\rm Re} \left({\rm Corr}_{\rho,\Phi}^{f,g}(A,B)\right)\right|^2 \leq {\rm I}_{\rho,\Phi}^{f,g}(A) {\rm I}_{\rho,\Phi}^{f,g}(B).
\end{eqnarray*}
\end{corollary}
The next result, which is a generalization of inequality \eqref{semic} for tracial conditional expectations can be derived from Corollary \ref{corcor} (in the case that $\mathcal{A}$ and $\mathcal{B}$ are von Neumann algebras) but we prove it in a different fashion. \\
Let $\mathcal{B}$ be a $C^*$-subalgebra of $C^*$-algebra $\mathcal{A}$. If $\mathcal{E}:\mathcal{A}\to \mathcal{B}$ is a tracial conditional expectation, then ${\rm ran} (\mathcal{E}) \subseteq \mathcal{Z}(\mathcal{B})$. Indeed, if $A\in \mathcal{A}$ and $B\in \mathcal{B}$, then we have
\begin{eqnarray}\label{concom}
B \mathcal{E}(A) = \mathcal{E}(BA)=\mathcal{E}(AB)= \mathcal{E}(A)B.
\end{eqnarray}

If $(\mathcal{X},\langle \cdot ,\cdot \rangle)$ is a semi-inner product module over a $C^*$-algebra $\mathcal{A}$, then
the Cauchy--Schwarz inequality, for $x,y\in \mathcal{X}$, asserts that $\langle x,y\rangle \langle y,x\rangle\leq \|\langle y,y\rangle \| \langle x,x\rangle$ (see \cite{lanc}). 
In the case that $\langle y,y\rangle \in \mathcal{Z}(\mathcal{A})$, where $\mathcal{Z}(\mathcal{A})$ is the center of the $C^*$-algebra $\mathcal{A}$, the latter inequality turns into (see \cite{il1})
 \begin{eqnarray}\label{cuachysharp}
 \langle x,y\rangle \langle y,x\rangle\leq \langle y,y\rangle \langle x,x\rangle .
\end{eqnarray}

\begin{theorem}
Let $\mathcal{A}$ be a $C^*$-algebra, and let $\mathcal{B}$ be $C^*$-subalgebra of $\mathcal{A}$. If $\mathcal{E}:\mathcal{A}\to \mathcal{B}$ is a tracial conditional expectation, then
\begin{eqnarray*}
\left|{\rm Corr}_{\rho,\mathcal{E}}^{\prime f,g}(A,B)\right|^2=\left|{\rm Re} \left({\rm Corr}_{\rho,\mathcal{E}}^{{f,g}}(A,B)\right)\right|^2 \leq {\rm I}_{\rho,\mathcal{E}}^{{f,g}}(A) {\rm I}_{\rho,\mathcal{E}}^{{f,g}}(B)
\end{eqnarray*}
for all self-adjoint elements $A,B \in \mathcal{A}$ and all positive elements $\rho \in \mathcal{A}$. In particular, if $f =g$ , then the above inequality holds for all normal elements $A,B \in \mathcal{A}$.
\end{theorem}
\begin{proof}
Let us define the $\mathcal{B}$-valued map $\langle\cdot , \cdot \rangle :\mathcal{A} \times \mathcal{A} \to \mathcal{B}$ by $\langle A ,B \rangle={\rm Corr}_{\rho,\mathcal{E}}^{\prime f,g}(A,B)$.
If $A,B \in \mathcal{A}$ and $C\in \mathcal{B}$, then
\begin{align*}
\langle A ,BC \rangle &= {\rm Corr}_{\rho,\mathcal{E}}^{\prime f,g}(A,BC)\\ &=\frac{1}{2} \left({\rm Corr}_{\rho,\mathcal{E}}^{{f,g}}(A,BC)+{\rm Corr}_{\rho,\mathcal{E}}^{{f,g}}(C^*B^*,A^*)\right)\\ &= \frac{1}{2}\Big(\mathcal{E}(f(\rho)g(\rho) A^*BC)-\mathcal{E}(f(\rho)A^* g(\rho) BC) +\mathcal{E}(f(\rho) g(\rho) BCA^*) \\ & \ \ -\mathcal{E}(f(\rho)BC g(\rho) A^*)\Big)\\ &= \frac{1}{2}\Big(\mathcal{E}(f(\rho) g(\rho)A^*BC)-\mathcal{E}(f(\rho)A^* g(\rho) BC) +\mathcal{E}(CA^*f(\rho)g(\rho) B)\\ & \ -\mathcal{E}(C f(\rho) A^* g(\rho) B)\Big)
\\ & \tag{since $\mathcal{E}$ is tracial} \\ &= \frac{1}{2}\Big(\mathcal{E}(f(\rho) g(\rho) A^*B)C-\mathcal{E}(f(\rho)A^* g(\rho) B)C +\mathcal{E}(f(\rho)g(\rho) BA^*)C\\ &\ \ -\mathcal{E}(f(\rho) B g(\rho) A^*)C\Big) \tag{by equality (\ref{concom})}\\ &= {\rm Corr}_{\rho,\mathcal{E}}^{\prime f,g}(A,B)C \\ &= \langle A ,B \rangle C.
\end{align*}
Therefore, according to Definition \ref{semi}, we see that $(\mathcal{A},\langle\cdot , \cdot \rangle)$ is a semi-inner product $\mathcal{B}$-module. Furthermore, equality (\ref{concom}) implies that ${\rm ran}(\mathcal{E})$ is a subset of the center of $\mathcal{A}$ . Let $A$ and $B$ be self-adjoint elements in $\mathcal{A}$. Then we get
\begin{align*}
\left|{\rm Re} ({\rm Corr}_{\rho,\mathcal{E}}^{{f,g}}(A,B))\right|^2&= \left|\dfrac{1}{2}\left( {\rm Corr}_{\rho,\mathcal{E}}^{{f,g}}(A,B)+{\rm Corr}_{\rho,\mathcal{E}}^{{f,g}}(B,A)\right)\right|^2 \\ &= \left| {\rm Corr}_{\rho,\mathcal{E}}^{\prime f,g}(A,B)\right|^2\\ &= \left| \langle A ,B \rangle \right|^2 \\ &\leq \langle A ,A \rangle \langle B,B\rangle \tag{by inequality (\ref{cuachysharp}} \\ &= {\rm I}_{\rho,\mathcal{E}}^{{f,g}}(A) {\rm I}_{\rho,\mathcal{E}}^{{f,g}}(B) \tag{since $A$ and $B$ are self-adjoint}.
\end{align*}
Finally, it is easy to see that if $f=g$ and $A, B$ are normal operators, then ${\rm I}_{\rho,\mathcal{E}}^{{f,g}}(A)= {\rm I}_{\rho,\mathcal{E}}^{{f,g}}(A^*)$ and ${\rm I}_{\rho,\mathcal{E}}^{{f,g}}(B)={\rm I}_{\rho,\mathcal{E}}^{{f,g}}(B^*)$.
\end{proof}
\section{Some relations between covariance and correlation}\label{sec5}
In this section, we give some relations between covariance and correlation.

\begin{theorem}\label{mogh}
Let $\Phi$ be a tracial positive linear map from a $C^*$-algebra $\mathcal{A}$ into a unital  $C^*$-algebra $\mathcal{B}$. Then
\begin{align*}
&\hspace{-1cm}\dfrac{1}{2} \big(\Phi (f(\rho) A^* g(\rho) A)+\Phi (f(\rho) A g(\rho) A^*) \big) \\
&\geq \Phi\left(f(\rho)^{\frac{1}{2}} g(\rho)^{\frac{1}{2}} A^* f(\rho)^{\frac{1}{2}} g(\rho)^{\frac{1}{2}} A\right)\\
&\geq \Phi \left(f(\rho) A^* g(\rho)\right) \Phi \Big(f(\rho)g(\rho)\Big)^{-1} \Phi \left(g(\rho) A f(\rho)\right)
\end{align*}
for all $\rho\in \mathcal{A}_h$ and all continuous positive real-valued functions $f$ and $g$ on the spectrum $\rho$ with $f(\rho)g(\rho)>0$ and all operators $A\in\mathcal{A}$. In particular,
\begin{align*}
{\rm I}_{\rho, {\Phi}}^{f,g}(A)\leq {\rm I}_{\rho, {\Phi}}^{\sqrt{fg},\sqrt{fg}}(A) \leq {\rm Var}_{\rho,\Phi}^{fg,1}(A)
\end{align*}
for every $A\in\mathcal{A}_h$
\end{theorem}
\begin{proof}
Let $X=g(\rho)^{\frac{1}{2}}A f(\rho)^{\frac{1}{2}} -f(\rho)^{\frac{1}{2}} A g(\rho)^{\frac{1}{2}}.$
Then we have
\begin{align*}
 0 &\leq \Phi(X^*X)\\ &= \Phi \Big(\big( f(\rho)^{\frac{1}{2}}A^* g(\rho)^{\frac{1}{2}} -g(\rho)^{\frac{1}{2}} A^* f(\rho)^{\frac{1}{2}}\big)\big(g(\rho)^{\frac{1}{2}}A f(\rho)^{\frac{1}{2}} -f(\rho)^{\frac{1}{2}} A g(\rho)^{\frac{1}{2}}\big) \Big) \\ &= \Phi \big( f(\rho)^{\frac{1}{2}}A^* g(\rho)^{\frac{1}{2}} g(\rho)^{\frac{1}{2}}A f(\rho)^{\frac{1}{2}}\big) -\Phi \big( f(\rho)^{\frac{1}{2}}A^* g(\rho)^{\frac{1}{2}} f(\rho)^{\frac{1}{2}}A g(\rho)^{\frac{1}{2}}\big)\\ & \ \ -\Phi \big( g(\rho)^{\frac{1}{2}}A^* f(\rho)^{\frac{1}{2}} g(\rho)^{\frac{1}{2}}A f(\rho)^{\frac{1}{2}}\big) + \Phi \big( g(\rho)^{\frac{1}{2}}A^* f(\rho)^{\frac{1}{2}} f(\rho)^{\frac{1}{2}}A g(\rho)^{\frac{1}{2}}\big);
\end{align*}
since $\Phi$ is tracial, we get
\begin{eqnarray*}
\Phi (f(\rho) A^* g(\rho) A)+\Phi (f(\rho) A g(\rho) A^*)\geq 2\Phi\left(f(\rho)^{\frac{1}{2}} g(\rho)^{\frac{1}{2}} A^* f(\rho)^{\frac{1}{2}} g(\rho)^{\frac{1}{2}} A\right),
\end{eqnarray*}
which implies the first inequality.
Furthermore, the matrix
\begin{eqnarray*} \begin{bmatrix} f(\rho)^{\frac{1}{4}}g(\rho)^{\frac{1}{4}}A^* g(\rho)^{\frac{1}{2}} f(\rho)^{\frac{1}{2}} A g(\rho)^{\frac{1}{4}}f(\rho)^{\frac{1}{4}} & f(\rho)^{\frac{1}{4}}g(\rho)^{\frac{1}{4}}A^* g(\rho)^{\frac{3}{4}} f(\rho)^{\frac{3}{4}} \\ g(\rho)^{\frac{3}{4}} f(\rho)^{\frac{3}{4}} A f(\rho)^{\frac{1}{4}}g(\rho)^{\frac{1}{4}} & f(\rho)g(\rho)
\end{bmatrix} \end{eqnarray*}
is positive, since
\begin{align*}
& f(\rho)^{\frac{1}{4}}g(\rho)^{\frac{1}{4}}A^* g(\rho)^{\frac{1}{2}} f(\rho)^{\frac{1}{2}} A g(\rho)^{\frac{1}{4}}f(\rho)^{\frac{1}{4}} \\
 &= f(\rho)^{\frac{1}{4}}g(\rho)^{\frac{1}{4}}A^* g(\rho)^{\frac{3}{4}} f(\rho)^{\frac{3}{4}} (f(\rho)g(\rho))^{-1} g(\rho)^{\frac{3}{4}} f(\rho)^{\frac{3}{4}} A f(\rho)^{\frac{1}{4}}g(\rho)^{\frac{1}{4}}.
\end{align*}
Using the two-positivity of $\Phi$, we assert that the matrix
\begin{eqnarray*} \begin{bmatrix} \Phi \left(f(\rho)^{\frac{1}{4}}g(\rho)^{\frac{1}{4}}A^* g(\rho)^{\frac{1}{2}} f(\rho)^{\frac{1}{2}} A g(\rho)^{\frac{1}{4}}f(\rho)^{\frac{1}{4}} \right) & \Phi \left( f(\rho)^{\frac{1}{4}}g(\rho)^{\frac{1}{4}}A^* g(\rho)^{\frac{3}{4}} f(\rho)^{\frac{3}{4}}\right) \\ \Phi \left(g(\rho)^{\frac{3}{4}} f(\rho)^{\frac{3}{4}} A f(\rho)^{\frac{1}{4}}g(\rho)^{\frac{1}{4}}\right) & \Phi \Big(f(\rho)g(\rho)\Big)
\end{bmatrix} \end{eqnarray*}
is positive. Hence, by applying Lemma \ref{matrixp} again, we arrived at the second inequality.
\end{proof}

\begin{corollary}
If $\Phi$ is a tracial positive linear map from a $C^*$-algebra $\mathcal{A}$ into a unital  $C^*$-algebra $\mathcal{B}$ and $\rho$ is a positive operator, then
\begin{eqnarray*}
 {\rm I}_{\rho, {\Phi}}^\alpha(A) \leq {\rm I}_{\rho,\Phi}^{(\frac{1}{2})}(A) \leq {\rm Var}_{\rho,\Phi}(A)
\end{eqnarray*}
for every self-adjoint $A\in\mathcal{A}$.
\end{corollary}
\begin{proof}
Using Theorem \ref{mogh} for $f(x)= x^\alpha$ and $g(x)=x^{1-\alpha}$, we get
\begin{align*}
 {\rm I}_{\rho, \phi}^\alpha(A) &=\Phi(\rho A^2)-\Phi(\rho^{\alpha}A\rho^{1-\alpha}A) \\ &\leq \Phi(\rho A^2)-\Phi(\rho^{\frac{\alpha}{2}}\rho^{\frac{1-\alpha}{2}}A\rho^{\frac{\alpha}{2}}\rho^{\frac{1-\alpha}{2}}A) \\ & \tag{by the first inequality in Theorem \ref{mogh}} \\ &= \Phi(\rho A^2)-\Phi(\rho^{\frac{1}{2}}A\rho^{\frac{1}{2}}A) ={\rm I}_{\rho,\Phi}^{(\frac{1}{2})}(A) \\ &\leq
 \Phi(\rho A^2)- \Phi(\rho^{\alpha} A \rho^{1-\alpha}) \Phi(\rho)^{-1} \Phi(\rho^{\alpha} A \rho^{1-\alpha})\\ & \tag{by the second inequality in Theorem \ref{mogh}}\\ &= \Phi(\rho A^2)- \Phi(\rho A) \Phi(\rho)^{-1} \Phi(\rho A) = {\rm Var}_{\rho,\Phi}(A).
\end{align*}
\end{proof}

\end{document}